\documentclass[10pt]{article}
\usepackage[a4paper]{geometry}
\usepackage{amsfonts}
\usepackage{amsmath}
\usepackage{amssymb}
\usepackage{amsthm}
\usepackage{bm}
\usepackage{xcolor}
\usepackage{faktor}
\usepackage{graphicx}
\usepackage{physics}
\usepackage{enumitem}
\usepackage{hyperref}
\hypersetup{linktoc=page}
\usepackage[capitalize, noabbrev]{cleveref}

\setitemize[0]{label=---}

\newcommand{\la}{\langle}
\newcommand{\ra}{\rangle}

\newcommand{\C}{{\mathbb C}}
\newcommand{\R}{{\mathbb R}}

\newcommand{\Z}{{\mathbb Z}}

\newcommand{\Chi}{{\mathcal X}}
\newcommand{\LL}{{\mathfrak L}}

\newtheorem{theorem}{Theorem}
\newtheorem{thmx}{Theorem}

\newtheorem{lemma}{Lemma}
\newtheorem{prop}{Proposition}
\newtheorem{cor}{Corollary}

\newtheorem{rem}{Remark}

\title{The Lie algebra generated by gradient vector fields}
\author{Bettina Kazandjian\footnote{Sorbonne Université, Université Paris Cité, CNRS, Inria, Laboratoire Jacques-Louis Lions, Paris, France, \href{mailto:bettina.kazandjian@sorbonne-universite.fr}{\texttt{bettina.kazandjian@sorbonne-universite.fr}}}, Omar Mohsen\footnote{Université Paris Cité, Sorbonne Université, CNRS, IMJ-PRG, F-75013 Paris, France, \href{mailto:omar.mohsen@imj-prg.fr}{\texttt{omar.mohsen@imj-prg.fr}}}, Eugenio Pozzoli\footnote{Univ Rennes, CNRS, IRMAR - UMR 6625, F-35000 Rennes, France, \href{mailto:eugenio.pozzoli@univ-rennes.fr}{\texttt{eugenio.pozzoli@univ-rennes.fr}}}}
\date{}

\begin{document}

\maketitle
\begin{abstract}
Let $M$ be a smooth compact manifold, and $g$ a smooth non-degenerate symmetric bilinear form on $TM$. We prove that every smooth vector field can be written as a finite linear combination of iterated Lie brackets of gradient vector fields.
\end{abstract}


\textbf{MSC Codes 2020:} 17B66, 47F10, 58D05.

\section{Introduction}

Let $M$ be a smooth compact manifold, $g$ a smooth non-degenerate symmetric bilinear form on $TM$. The set of real smooth vector fields on $M$ is denoted by $\Chi(M,\R)$. For any $f \in C^{\infty}(M,\R)$, we define the gradient $\nabla_gf$ of $f$ with respect to $g$ by the formula $\mathrm{d}f_x(v)=g_x(\nabla_gf(x),v)$ for every $v\in T_xM$. The set of gradient vector fields is not closed under Lie brackets, when $g$ is symmetric. For every $m\in \Z_+$, the vector space spanned (over $\mathbb{R}$) by Lie brackets of depth lower than or equal to $m$ of gradient vector fields is denoted by $\LL_g^m$. More precisely, $\LL_g^0:=\left\{\nabla_gf \mid f\in C^\infty(M,\R)\right\}$ and
\begin{equation*}
    \LL_g^m:=\left\{X+[\nabla_gf,Y]\mid f\in C^{\infty}(M,\R),X,Y\in \LL_g^{m-1}\right\}, \qquad  m \in \Z_+^*.
\end{equation*}
The Lie algebra generated by gradient vector fields is denoted by
\begin{equation*}
    \LL_g:=\mathrm{Lie}\left\{\nabla_gf\mid f \in C^\infty(M,\R)\right\}=\bigcup_{m\in \mathbb{Z}_+}\LL_g^m.
\end{equation*}
A natural question is whether the Lie algebra $\LL_g$ coincides with the space of smooth vector fields. The main result of this article provides a positive answer to such question.
\begin{thmx}\label{theorem : main theorem}
    There exists an integer $m\in \Z_+$ such that $\LL_g^m=\Chi(M,\R)$, that is, every smooth vector field can be written as a linear combination of iterated Lie brackets of gradient vector fields with a maximum depth of $m$.
\end{thmx}
We note that Theorem \ref{theorem : main theorem} is in strong contrast to the case where $g$ is antisymmetric. For example, if $g$ is a symplectic form, then the space of gradient vector fields (usually called Hamiltonian vector fields in this case) is closed under Lie brackets, and thus forms a proper Lie sub-algebra  of $\Chi(M,\R)$.

The set of $C^\infty$-diffeomorphisms of $M$ connected to the identity is denoted by ${\rm Diff_0}(M)$. We endow it with the compact-open topology.

Let $\mathcal{G}_g$ be the group generated by $\exp(\LL_g^0)$, that is, the group whose elements are compositions of flows of gradient vector fields. More precisely,
$$\mathcal{G}_g:=\{\exp(\nabla_g f_n)\circ\dots\circ \exp(\nabla_g f_1)\mid f_1,\dots,f_n\in C^\infty(M,\R), n\in \Z_+\},$$
where $\exp(t X)$ denotes the flow at time $t\in \R$ of the ODE $\dot{x}(t)=X(x(t)), x\in M, X\in \Chi(M,\R)$. We have the following consequence of Theorem \ref{theorem : main theorem}.
\begin{cor}\label{cor : main corollary}
The group $\mathcal{G}_g$ is dense in ${\rm Diff_0}(M)$.
\end{cor}
From a geometric control viewpoint, the Lie algebra $\LL_g$ can be seen as the space of admissible directions for the infinite-dimensional control system, where the state is a diffeomorphism of $M$ and every gradient vector field is a directly accessible control field. More precisely, $\LL_g$ is the space of admissible directions for the flow of the following non-autonomous ODE
%Such Lie algebra appears in evolution problems where dynamics are directly propagated along arbitrary gradient vector fields, hence indirectly also along linear combination of their iterated Lie brackets. It thus describes the space of admissible directions for non-autonomous dynamics of the form 
\begin{equation}\label{eq:control}
\dot{x}(t)=\nabla_g f(x(t),t), \quad x\in M, t\geq 0,
\end{equation}
where $(t\mapsto f(\cdot,t))\in {\rm PWC}([0,\infty), C^\infty(M,\R))$ is a piecewise constant-in-time control law which can select at any time an arbitrary real-valued smooth function over $M$. Such control problem arise e.g. in quantum mechanical systems modelled through Schrödinger PDEs. Indeed, in \cite{beauchard-pozzoli2} it is proved that, if the potential is rich enough, the wavefunction can be transported along the flow generated by every gradient vector field. The group $\mathcal{G}_g$ describes then the space of reachable configurations for the flow of \eqref{eq:control}, and Corollary \ref{cor : main corollary} is equivalent to the approximate controllability in the group of diffeomorphisms of control system \eqref{eq:control}.

\subsection{Bibliographic comments}
An approximate version, for flat tori and Euclidean spaces, of Theorem \ref{theorem : main theorem} is established in \cite[Propositions 31 \& 32]{beauchard-pozzoli2}. More precisely, it is proved that every smooth vector field on the flat torus or Euclidean space can be $C^1$-approximated by linear combinations of iterated Lie brackets of a finite family of gradient vector fields.

Concerning Corollary \ref{cor : main corollary}, we mention the article \cite{berger}, where a similar question is studied. In particular, the group generated by vertical and horizontal shears is proved to be dense in the group of Hamiltonian diffeomorphisms of tori and Euclidean spaces. Vertical and horizontal shears are two abelian groups defined as flows of symplectic gradient vector fields of functions depending either only on position or momentum variables.

With respect to those previous contributions, the novelty of the present article consists in proving the results, previously established on flat tori and Euclidean spaces, on every compact manifold equipped with a non-degenerate symmetric bilinear form. This is achieved by reformulating the question as an ellipticity problem of differential operators.

We also notice that Corollary \ref{cor : main corollary} may alternatively be deduced from Theorem \ref{theorem : main theorem} by applying \cite[Theorem D]{Liu}, but we chose to derive it in a slightly more constructive way.

A further question, that remains open, is whether the group $\mathcal{G}_g$ is equal to ${\rm Diff_0}(M)$. We mention the article \cite{agrachev-caponigro} (and also \cite{trelat-2017} for related investigations in Sobolev spaces), where a similar question is studied. More precisely, it is proved that the group generated by the flows of any family of vector fields, that satisfies Hörmander's condition and that is invariant under multiplication by smooth functions, is equal to ${\rm Diff_0}(M)$. Notice that those hypothesis do not apply to our setting, since gradient vector fields are not invariant under multiplication by smooth functions.

\medskip

The article is organized as follows. In Section \ref{sec:outline} we outline the scheme of the proof of Theorem \ref{theorem : main theorem}. In Section \ref{sec:elliptic} we recall some needed facts on elliptic operators. Section \ref{sec:core} contains the core of the proof of Theorem \ref{theorem : main theorem}. We conclude in Section \ref{sec:corollary} by proving Corollary \ref{cor : main corollary}.

\section{Proof of Theorem \ref{theorem : main theorem}}\label{sec:outline}
The proof of Theorem \ref{theorem : main theorem} is based on a theorem of J. Grabowski \cite{grabowski}, and \cref{theorem : finite codimension}.

\begin{theorem}[J. Grabowski] \label{theorem : Grabowski}
    Let $L\subsetneq\Chi(M,\R)$ be a maximal Lie algebra of finite codimension. Then there exists $x\in M$ such that $L=\{X\in \Chi(M,\R)\mid X(x)=0\}$.
\end{theorem}

\begin{prop} \label{theorem : finite codimension}
    The vector space $\LL_g^1$ has finite codimension in $\Chi(M,\R)$.
\end{prop}

\begin{proof}[Proof of \cref{theorem : main theorem}]
    Since $\LL_g^1 \subset \LL_g$, by \cref{theorem : finite codimension}, $\LL_g$ has also finite codimension in $\Chi(M,\R)$.
    If $\LL_g\neq \Chi(M,\R)$, then there exists a maximal Lie subalgebra $\LL_g\subseteq L\subsetneq \Chi(M,\R)$.
    So, by Grabowski's theorem, there exists $x\in M$ such that for every $X\in \LL_g$, $X(x)=0$.
    This is a contradiction because for every $x\in M$ there exists $f \in C^\infty(M,\R)$ such that $\mathrm{d}f_x\neq 0$. Then necessarily $\nabla_gf(x)\neq 0$, and $\nabla_gf\in \LL_g$.
    Hence, $\LL_g=\Chi(M,\R)$.
    Let us denote by $H \subset \Chi(M,\R)$ a finite dimensional supplementary of $\LL_g^1$ in $\Chi(M,\R)$, that is, $\Chi(M,\R)=\LL_g^1\oplus H$. We have already proved that $\Chi(M,\R)=\LL_g$, so $H\subset \LL_g$ and necessarily there exists $m\in \Z_+$ such that $H \subset \LL_g^m$ because $H$ is finite dimensional. Then $\Chi(M,\R)=\LL_g^1\oplus H \subseteq \LL_g^{\mathrm{max}(1,m)}$.
\end{proof}

The proof of \cref{theorem : finite codimension} relies on a well-known result from the theory of elliptic operators which we now recall.

\section{Elliptic operators}\label{sec:elliptic}

Let $\ell,m \in \Z_+$, $D : C^{\infty}(M,\C^\ell) \rightarrow \Chi(M,\C)$ a linear differential operator of order $m$. We denote by $\sigma^m(D,x,\xi):\C^\ell\to T_xM\otimes \C$ its classical principal symbol at $(x,\xi)\in T^*M\backslash 0$.
Let us recall a consequence of the elliptic regularity theorem, see \cite[Chapter 18]{hormander}.
\begin{theorem}
    If the principal symbol $\sigma^m(D,x,\xi): \C^\ell \rightarrow T_xM\otimes \C$ is bijective for every $(x,\xi)\in T^*M\backslash 0$, then $D$ is elliptic and the image of $D$ denoted by $\mathrm{Im}(D)$ has finite codimension in $\Chi(M,\C)$.
\end{theorem}

\begin{prop} \label{theorem : principal symbol surjective}
    Let $D : C^{\infty}(M,\C^\ell)\rightarrow \Chi(M,\C)$ be a differential operator of order $m\in \Z_+$. If its principal symbol $\sigma^m(D,x,\xi): \C^\ell \rightarrow T_xM\otimes \C$ is surjective for every $(x,\xi)\in T^*M \backslash 0$, then $\mathrm{Im}(D)$ has finite codimension in $\Chi(M,\C)$.
\end{prop}
\begin{proof}
    If $\sigma^m(D)$ is surjective, then $\sigma^m(D)\sigma^m(D)^*=\sigma^m(DD^*)$ is bijective. So $DD^*$ is elliptic, and so $\mathrm{Im}(DD^*)$ has finite codimension in $\Chi(M,\C)$. But $\mathrm{Im}(DD^*)\subseteq \mathrm{Im}(D)$, so $\mathrm{Im}(D)$ also has finite codimension in $\Chi(M,\C)$.
\end{proof}

\begin{rem}
    Suppose in addition to the hypothesis of \Cref{theorem : principal symbol surjective} that the coefficients of $D$ are real-valued, i.e., $D(\overline{f})=\overline{D(f)}$ for all $f\in C^{\infty}(M,\C^\ell)$.
    In this case, $D(C^{\infty}(M,\R^\ell))\subset \Chi(M,\R)$, and $D(C^{\infty}(M,\R^\ell))$ has finite codimension in $\Chi(M,\R)$.
    %Notice also that in this case, the principal symbol is a map $\sigma^m(D)(x,\xi):\R^\ell\to T_xM$.
\end{rem}
We will now recall the Direct sum construction of differential operators. Suppose that we have two linear differential operators of order $m$ with real coefficients $D_1:C^{\infty}(M,\R^{\ell_{1}}) \rightarrow \Chi(M,\R)$ and $D_2:C^{\infty}(M,\R^{\ell_{2}}) \rightarrow \Chi(M,\R)$.
We can then define their direct sum
\begin{equation*}\begin{aligned}
        D_1\oplus D_2:C^\infty(M,\R^{\ell_1+\ell_2})\to \Chi(M,\R),\quad D_1\oplus D_2(f_1,f_2)=D_1(f_1)+D_2(f_2).
    \end{aligned}\end{equation*}
The differential operator $D_1\oplus D_2$ has real coefficients, and its principal symbol is given by
\begin{equation*}\begin{aligned}
        \sigma^m(D_1\oplus D_2,x,\xi)(\lambda,\mu)=\sigma^m(D_1,x,\xi)(\lambda)+  \sigma^m(D_2,x,\xi)(\mu) ,\quad \lambda\in \C^{\ell_1},\mu\in \C^{\ell_2}.
    \end{aligned}\end{equation*}
So, the image of the principal symbol of $D_1\oplus D_2$ is the sum of the images of the principal symbols of $D_1$ and $D_2$.
\section{Proof of \cref{theorem : finite codimension}}\label{sec:core}

In this section, we prove the following statement, from which \cref{theorem : finite codimension} follows as a corollary, according to \cref{theorem : principal symbol surjective}.

\begin{theorem}\label{theorem : elliptic construction}
    There exist $\ell \in \Z_+$ and $D: C^\infty(M,\R^\ell)\rightarrow \Chi(M,\R)$ a differential operator with real coefficients of order $1$ such that $\mathrm{Im}(D)\subset \LL_g^1$ and $\sigma^1(D,x,\xi):\C^\ell\to T_xM\otimes \C$ is surjective for all $(x,\xi)\in T^*M\backslash 0$.
\end{theorem}
\begin{lemma}\label{lem : symbol at point}
    For every $(x,\xi)\in T^*M\backslash 0$, there exist $\ell \in \Z_+$ and $D: C^\infty(M,\R^\ell)\rightarrow \Chi(M,\R)$ a differential operator with real coefficients of order $1$ such that $\mathrm{Im}(D)\subset \LL_g^1$ and $\sigma^1(D,x,\xi)$ is surjective.
\end{lemma}
\begin{proof}[Proof of \Cref{theorem : elliptic construction}]
    By continuity of the principal symbol, if $\sigma^1(D,x,\xi)$ is surjective at $(x,\xi)$, then it is surjective in a neighborhood of $(x,\xi)$.
    So, by compactness of the sphere cotangent bundle, we can find a finite number of differential operators
    \begin{equation*}\begin{aligned}
            D_{1}: C^\infty(M,\R^{\ell_1})\rightarrow \Chi(M,\R),\dots, D_{n}: C^\infty(M,\R^{\ell_n})\rightarrow \Chi(M,\R)
        \end{aligned}\end{equation*}
    such that $\mathrm{Im}(D_{i})\subset \LL_g^1$ for all $i$, and for every $(x,\xi)\in T^*M\backslash 0$, there exists $i$ such that $\sigma^1(D_{i},x,\xi)$ is surjective.
    We can then define
    the differential operator
    \begin{equation*}\begin{aligned}
            D=D_1\oplus \cdots\oplus D_n:C^\infty(M,\R^{\ell_1+\cdots+\ell_n})\rightarrow \Chi(M,\R),
        \end{aligned}\end{equation*}
    which satisfies
    $\mathrm{Im}(D)\subset \LL_g^1$, and for every $(x,\xi)\in T^*M\backslash 0$, $\sigma^1(D,x,\xi)$ is surjective.
\end{proof}

\begin{proof}[Proof of \cref{lem : symbol at point}.]
    Let $f\in C^{\infty}(M,\R)$. We introduce the differential operator
    \begin{equation*}\begin{aligned}
            D_f : C^{\infty}(M,\R)\rightarrow \Chi(M,\R),\quad D_f( h) =[\nabla_g f,\nabla_g(fh)].
        \end{aligned}\end{equation*}
    First we note that $\mathrm{Im}D_f \subset \LL^1_g$. The differential operator $D_f$ is of order 2. Let us calculate the operator more explicitly by expanding the Lie bracket.
    For every $h\in C^{\infty}(M,\R)$, we have
    \begin{align*}
        D_f(h) & =[\nabla_gf,\nabla_g(fh)]                                                                                      \\
               & =[\nabla_gf,f\nabla_gh+h\nabla_gf]                                                                             \\
               & =[\nabla_g f,f\nabla_gh]+[\nabla_gf,h\nabla_gf]                                                                \\
               & =f[\nabla_g f,\nabla_gh]+\mathrm{d}f(\nabla_gf)\nabla_gh+h[\nabla_g f,\nabla_g f]+\mathrm{d}h(\nabla_g f)\nabla_g f              \\
               & =f[\nabla_g f,\nabla_gh]+\mathrm{d}f(\nabla_gf)\nabla_gh+\mathrm{d}h(\nabla_g f)\nabla_g f                                       \\
               & =[f\nabla_gf,\nabla_gh]+\mathrm{d}f(\nabla_gh)\nabla_gf+\mathrm{d}f(\nabla_gf)\nabla_gh+\mathrm{d}h(\nabla_g f)\nabla_g f                 \\
               & =\frac{1}{2}[\nabla_g(f^2),\nabla_g h]+\mathrm{d}f(\nabla_gh)\nabla_gf+\mathrm{d}f(\nabla_gf)\nabla_gh+\mathrm{d}h(\nabla_g f)\nabla_g f.
    \end{align*}
    Moreover, by the symmetry of $g$,
    \begin{equation*}
        \mathrm{d}f(\nabla_gh)=g(\nabla_gf,\nabla_gh)=g(\nabla_gh,\nabla_g f)=\mathrm{d}h(\nabla_gf).
    \end{equation*}
    So,
    \begin{equation*}
        D_f(h)=\frac{1}{2}[\nabla_g(f^2),\nabla_gh]+2\mathrm{d}h(\nabla_gf)\nabla_gf+\mathrm{d}f(\nabla_g f)\nabla_gh.
    \end{equation*}

    \begin{rem}
        If $g$ was anti-symmetric, for example if $g$ was a symplectic form, the term $\mathrm{d}f(\nabla_gh)\nabla_gf+\mathrm{d}h(\nabla_gf)\nabla_gf$ would vanish. This is the term which ultimately gives surjectivity of the symbol.
    \end{rem}
    Now, suppose that we have a finite family of functions $(f_j)_{j\in J}\subseteq C^{\infty}(M,\R)$, where $J\subseteq \Z_+$ has a finite cardinality, such that $\sum_{j\in J}f_j^2=1$.
    Then, if we consider the differential operator
    \begin{equation*}\begin{aligned}
            \sum_{j\in J}D_{f_j}: C^{\infty}(M,\R)\rightarrow\Chi(M,\R),\quad (\sum_{j\in J}D_{f_j})(h)=\sum_{j\in J}D_{f_j}(h),
        \end{aligned}\end{equation*}
    we have
    \begin{align*}
        \sum_{j\in J}D_{f_j}(h) & =\frac{1}{2}[\nabla_g\left(\sum_{j\in J}f_j^2\right),\nabla_gh]+\sum_{j\in J}2\mathrm{d}h(\nabla_gf_j)\nabla_gf_j+\mathrm{d}f_j(\nabla_gf_j)\nabla_gh \\
                            & =\sum_{j\in J}2\mathrm{d}h(\nabla_gf_j)\nabla_gf_j+g(\nabla_gf_j,\nabla_gf_j)\nabla_gh.
    \end{align*}
    So $\sum_{j\in J}D_{f_j}$ is of order less than or equal to $1$. Moreover $\mathrm{Im}(\sum_{j\in J}D_{f_j})\subset \LL_g^1$. Let us compute its principal symbol of order 1. For every $(x,\xi)\in T^*M\backslash 0$,
    \begin{equation*}
        \sigma^1(\sum_{j\in J}D_{f_j},x,\xi)=\sum_{j\in J}2\la \xi, \nabla_g f_j(x)\ra\nabla_g f_j(x)+g_x(\nabla_g f_j(x),\nabla_g f_j(x))\xi^{\flat},
    \end{equation*}
    where $\xi^\flat$ is the unique element of $T_xM$ such that $\la \xi,v\ra=g_x(\xi^\flat,v)$ for every $v\in T_xM$.
    \begin{lemma}\label{lemma : parition}
        For every $x_0\in M$, $X\in T_{x_0}M\backslash \{0\}$, there exists a finite family of functions $f_1,\cdots,f_n\in C^\infty(M,\R)$ such that $\sum_{j=1}^nf_j^2=1$, $\nabla_gf_1(x_0)=X$ and $\nabla_gf_j(x_0)=0$ if $j\neq1$.
    \end{lemma}
    \begin{proof}
        By duality with $g$, we need to prove that for every $x_0\in M$ and $\xi_0\in T^*M\backslash 0$,
        there exists a finite family of functions $f_1,\cdots,f_n\in C^\infty(M,\R)$ such that $\sum_{j=1}^nf_j^2=1$, $\mathrm{d}f_1(x_0)=\xi_0$ and $\mathrm{d}f_j(x_0)=0$.
        By a partition of unity argument, it is enough to define the functions $f_1,\cdots,f_n$ near $x_0$.
        Let $m=\dim(M)$.
        We can choose a local chart such that $M=\R^{m}$, $x_0=0$, $\xi_0=(1,0,\cdots,0)$.
        In this case, we take the functions $f_1(x_1,\cdots,x_m)=\sin(x_1),f_2(x_1,\cdots,x_m)=\cos(x_1)$.
    \end{proof}
    We now fix $(x_0,\xi_0)\in T^*M\backslash 0$ and $X\in T_{x_0}M\backslash \{0\}$. Let $\mathcal{D}_X:=\sum_{j=1}^nD_{f_j}:C^\infty(M,\R)\to \Chi(M,\R)$ where $(f_j)_{j=1,\cdots,n}$ is the family obtained from \cref{lemma : parition}. Its principal symbol at $(x_0,\xi_0)$ is given by
    \begin{equation*}
        \sigma^1(\mathcal{D}_X,x_0,\xi_0)=2\la \xi_0,X\ra X+g_{x_0}(X,X)\xi_0^\flat.
    \end{equation*}
    Finally, notice that the differential operator
    \begin{equation*}\begin{aligned}
            D':C^\infty(M,\R)\to \Chi(M,\R),\quad h\mapsto  \nabla_gh,
        \end{aligned}\end{equation*}
    is also of order $1$, and its principal symbol of order $1$ at $(x_0,\xi_0)$ is equal to $\xi^\flat_0$. Furthermore $\mathrm{Im}(D')\subseteq \LL_g^1$.

    We now choose any linear basis $X_1,\dots,X_m\in T_{x_0}M$ where $m=\dim(M)$ and such that $\langle \xi_0,X_1\rangle\neq 0$ and $\langle \xi_0,X_j\rangle=0, j=2,\dots,m$.
    Let $D$ be the direct sum of the differential operators $D'$, $\mathcal{D}_{X_1}$, and $\mathcal{D}_{X_1+X_j}$ for $j=2,\cdots,m$.
    The image of the principal symbol of $D$ at $(x_0,\xi_0)$ contains $\xi_0^\flat$, and  so it contains
    \begin{equation*}\begin{aligned}
            2\la \xi_0,X_1\ra X_1,\quad 2\la \xi_0,X_1+X_j\ra (X_1+X_j), \qquad j=2,\dots,m.
        \end{aligned}\end{equation*}
    Hence, the principal symbol of $D$ is surjective.
\end{proof}

\section{Proof of Corollary \ref{cor : main corollary}}\label{sec:corollary}
The proof of Corollary \ref{cor : main corollary} is based on a theorem of W. Thurston \cite{thurston} and Theorem \ref{theorem : main theorem}.
\begin{theorem}[W. Thurston]\label{thurston theorem}
The group ${\rm Diff_0}(M)$ is simple. 
\end{theorem}
It is a direct consequence of Theorem \ref{thurston theorem} that the group generated by $\exp(\Chi(M,\R))$ is equal to ${\rm Diff_0}(M)$, because it is normal. So every $\phi\in {\rm Diff_0}(M)$ can be written as $\phi=\exp(X_k)\circ\dots\circ\exp(X_1)$ for some $X_1,\dots,X_k\in \Chi(M,\R), k\in \Z_+$. In order to prove Corollary \ref{cor : main corollary}, it then suffices to show that every flow of vector field, i.e. every diffeomorphism of the form $\phi=\exp(X), X\in\Chi(M,\R)$, is the limit of a sequence of diffeomorphisms in $\mathcal{G}_g$. This follows from Theorem \ref{theorem : main theorem} and the product formulas, for $X,Y\in \Chi(M,\R)$,
\begin{align*}
&\lim_{n\to \infty} \left(\exp(X/n)\exp(Y/n)\right)^n=\exp(X+Y),\\
&\lim_{n\to \infty} \left(\exp(X/n)\exp(Y/n)\exp(-X/n)\exp(-Y/n)\right)^{n^2}=\exp([X,Y]).
\end{align*}

\textbf{Acknowledgments.} The authors are supported by ANR-25-CE40-4062 (project QUEST), ANR-24-CE40-3008-01 (project QuBiCCS), ANR-11-LABX-0020 (Centre Henri Lebesgue), and \newline ANR-23-CE40-0016 (project OpART).

\bibliographystyle{unsrt}
\bibliography{references_new}

@article{beauchard-Pozzoli2,
author = {Beauchard, K. and Pozzoli, E.},
title = {Small-time approximate controllability of bilinear {S}chrödinger equations and diffeomorphisms},
year        = {2025},
journal = {Annales de l'Institut Henri Poincaré Analyse Non-Linéaire},
note = {Published online first}
}

@article{Liu,
 author = {Heintze, Ernst and Liu, Xiaobo},
 title = {Homogeneity of infinite dimensional isoparametric submanifolds},
 fjournal = {Annals of Mathematics. Second Series},
 journal = {Ann. Math. (2)},
 issn = {0003-486X},
 volume = {149},
 number = {1},
 pages = {149--181},
 year = {1999},
 language = {English},
 doi = {10.2307/121022},
 url = {https://eudml.org/doc/120132},
 zbMATH = {1278977},
 Zbl = {0931.53027}
}

@article{berger,
 AUTHOR = {Berger, Pierre and Turaev, Dmitry},
     TITLE = {Generators of groups of {H}amiltonian maps},
   JOURNAL = {Israel J. Math.},
  FJOURNAL = {Israel Journal of Mathematics},
    VOLUME = {267},
      YEAR = {2025},
    NUMBER = {1},
     PAGES = {237--252}
}

@Article{trelat-2017,
 Author = {Arguill{\`e}re, Sylvain and Tr{\'e}lat, Emmanuel},
 Title = {Sub-{Riemannian} structures on groups of diffeomorphisms},
 FJournal = {Journal of the Institute of Mathematics of Jussieu},
 Journal = {J. Inst. Math. Jussieu},
 ISSN = {1474-7480},
 Volume = {16},
 Number = {4},
 Pages = {745--785},
 Year = {2017},

}

@article{agrachev-caponigro,
author = {Agrachev, Andrei and Caponigro, Marco},
year = {2009},
month = {11},
pages = {2503-2509},
title = {Controllability on the group of diffeomorphisms},
volume = {26},
journal = {Annales de l'Institut Henri Poincaré Analyse Non-Linéaire},
}

@article{grabowski,
 author = {Grabowski, Janusz},
 title = {Isomorphisms and ideals of the {Lie} algebras of vector fields},
 fjournal = {Inventiones Mathematicae},
 journal = {Invent. Math.},
 issn = {0020-9910},
 volume = {50},
 pages = {13--33},
 year = {1978}
}

@book{hormander,
 author = {H{\"o}rmander, Lars},
 title = {The analysis of linear partial differential operators. {III}: {Pseudo}-differential operators},
 edition = {Reprint of the 1994 ed.},
 fseries = {Classics in Mathematics},
 series = {Class. Math.},
 issn = {1431-0821},
 isbn = {978-3-540-49937-4},
 year = {2007},
 publisher = {Berlin: Springer}
}

@article{thurston,
  title={Foliations and groups of diffeomorphisms},
  author={William P. Thurston},
  journal={Bulletin of the American Mathematical Society},
  year={1974},
  volume={80},
  pages={304-307},
}

\end{document}